\documentclass[12pt]{article}
\usepackage{amsmath, amssymb}
\usepackage{color}

\newtheorem{theorem}{Theorem}

\newtheorem{lemma}{Lemma}

\newtheorem{definition}{Definition}

\newenvironment{proof}{\noindent\textit{Proof.}}{\hfill $\Box$\par\bigskip}

\title{The union-closed set conjecture is true}%
\author{Roberto Demontis \thanks{robdemontis@gmail.com}}%

\date{ }%

\begin{document}
\maketitle

\begin{abstract}
We prove that the conjecture made by Peter Frankl in the late 1970s is true. In other words for every finite union-closed family which contains a non-empty set, there is an element that belongs to at least half of its members.
\end{abstract}

The union-closed set conjecture is  one of the most famous combinatorial problems. It asserts that for every finite union-closed family which contains a non-empty set, there is an element that belongs to at least half of its members. In literature, there have been several partial results regarding this, as can be seen from the survey paper of Bruhn and Schaudt \cite{Bruhn}. For instance, in  \cite{Ale} V. B. Alekseev approximates the number of union-closed families of subset of $[n]$. In \cite{Balla} I. Balla , B. Bollobàs and T. Eccles prove that the conjecture occurs for union-closed families with more sets compared to the size of the universe, this size has been further improved  by Eccles in \cite{Ecc}.  On the other hand in \cite{Bos} I. Bo\textbf{$\check{s}$}njak  and P. Markovi\textbf{$\acute{c}$} show that the conjecture is true for any-closed family with a universe of 11 elements. Moreover in \cite{Fal} V. Falgas-Ravry proves that if a union-closed set with a universe of $n$ elements contains at most $2n$ element with the property to be separated, then the conjecture holds. Moreover in \cite{Mass} J. Massberg reinforces the latter result. 
\\

First of all we introduce some preliminary terminology and notation according to the literature.

 We call $\mathcal{A}= 2^{[n]}-\{ \emptyset\}$.
We will use uppercase letters for elements of $\mathcal{A}$ , lowercase letters for elements of  $\cup \mathcal{A}$, mathcal character uppercase letters for subsets  of  $\mathcal{A}$.

 Let $\mathcal{F} \subseteq \mathcal{A}$ union-closed if
for all $X, Y\in \mathcal{F}$, $X\cup Y \in\mathcal{F}$.

 For all $i \leq n$  we say $\mathcal{F}^i=\{X\in \mathcal{F}: i\in X\}$. 

We note $|\mathcal{A}|=2^n-1$ and $|\mathcal{A}^i|=2^{n-1}$.

The purpose of this paper is to prove

\begin{theorem}
 Let $\mathcal{F}\subseteq \mathcal{A}$ be a \textbf{union-closed set}, we can find $i \leq n$  such that  $$|\mathcal{F}|\leq2|\mathcal{ F}^i|.$$
\end{theorem}

Summarizing the structure of the paper we say that we want to imagine each union-closed set $\mathcal{F}$ obtained as the successive deletion of sets from the universe $\mathcal{A}$:

\begin{enumerate}
	\item Lemma 1, 2, 3 illustrate the properties of some elements of an union-closed set, that we will call basis of $\mathcal{F}$.
           \item Theorem 2 shows that it is possible to obtain each union-closed set $\mathcal{F}$ through successive deletions of individual elements of $\mathcal{A}$ so as to have only union-closed sets during the whole procedure.
            \item  Lemma 5 shows that we can get a sequence described in Theorem 2 in which the sets containing a certain element are eliminated only at the the end of the procedure.
           \item In Theorem 3 and Lemma 6 we create a sequence in which it is possibile to use induction.
           \item In Theorem 4 and Theorem 5 we prove a sentence that implies theorem 1.
\end{enumerate}

Now we start the proof:

\begin{definition}
The set $B(\mathcal{F})=\{X \in \mathcal{F} \ | \ \forall Y, Z \in \mathcal{F}- \{X\}: X\neq Y \cup Z\}$ is said \textbf{basis} of $\mathcal{F}$ .
\end{definition}

We  can prove

\begin{lemma}
For all $X \in \mathcal{F}$, not necessarily union-closed, we can always find a set  $\mathcal{T}=\{T_1;\dots; T_r\} \subseteq B(\mathcal{F})$, such that  $X = \cup \mathcal{T}$.
\end{lemma}
\begin{proof}
By induction on $|X|$ in $\mathcal{F}$.

By definition of basis, if $|X|= min_{Y\in \mathcal{F}}|Y|$,  we  must have $X \in  B(\mathcal{F})$, then we set $\mathcal{T}=\{X\}$ and so the Lemma is true.

If $|X|> min_{Y\in \mathcal{F}}|Y|$ we have two possibilities:
\begin{enumerate}
	\item $X \in  B(\mathcal{F})$, in this case the Lemma is true.
            \item  $X \notin B(\mathcal{F})$, in this case $X= Y\cup Z$, with $Y, Z\in \mathcal{F}$ and $|Y|, |Z| < |X|$, so the Lemma is true by induction.
\end{enumerate}
\end{proof}

Now we want to describe how to build a union-closed set $\mathcal{F}$ through successive eliminations of elements of $\mathcal{A}$.

 \begin{definition}
For each $\mathcal{F} \subseteq \mathcal{A}$ union-closed set, we say \textbf{sequence from $\mathcal{A}$ to $\mathcal{F}$} the sequence of  sets $\mathcal{A}_0, \mathcal{A}_1 \dots, \mathcal{A}_t$ such that:
\begin{enumerate}
	\item $\mathcal{A}_0= \mathcal{A}$,
           \item $\mathcal{A}_1= \mathcal{A}_0-\{X_1\}$
           \item  $\mathcal{A}_r= \mathcal{A}_{r-1}-\{X_r\}$
           \item $\mathcal{A}_t =\mathcal{A}_{t-1}-\{X_t\}= \mathcal{F}$.
          \end{enumerate}
\end{definition}
 
 \begin{definition}
For all union-closed sets $\mathcal{F} \subseteq \mathcal{A}$, we say \textbf{union-closed sequence from $\mathcal{A}$ to $\mathcal{F}$} the sequence $\mathcal{A}_0, \mathcal{A}_1 \dots,\mathcal{A}_t$ such that: $\forall i=1 \dots t $  $\mathcal{A}_i$ is a union-closed set.
\end{definition}

By the following two Lemmas, we will prove that for each $\mathcal{F} \subseteq \mathcal{A}$ union-closed set a union-closed sequence from $\mathcal{A}$ to $\mathcal{F}$ exists.

\begin{lemma}
Let $\mathcal{F}\subseteq \mathcal{A}$ be a union-closed set,  let $ B \in B(\mathcal{F})$,  $ \mathcal{F}-\{B\}$ is a union-closed set. 
\end{lemma}
\begin{proof}
It is sufficient to observe that by definition of basis for each $X, Y \in \mathcal{F}-\{B\}$  $ X\cup Y \neq B$. As a consequence $ X\cup Y \in \mathcal{F}-\{B\} $.
\end{proof}

\begin{lemma}
Let $\mathcal{F}\subseteq \mathcal{A}$ ($\mathcal{F}$ is not neccessarily union-closed)  let $ Z \in \mathcal{F}- B(\mathcal{F})$,  then $ B(\mathcal{F})\subseteq B(\mathcal{F}-\{Z\})$. 
\end{lemma}
\begin{proof}
 Let $ T \in B(\mathcal{F})$, for all $ X, Y \in \mathcal{F}-\{T\}$ we have $X\cup Y \neq T$.In other words we can claim that $ \forall X, Y \in\mathcal {F}-\{T; Z \}$  $X\cup Y \neq T$, as a consequence $T \in  B(\mathcal{F}-\{Z\})$.
\end{proof}

 \begin{theorem}
Let $\mathcal{F}\subseteq \mathcal{A}$ be a union-closed set, then it exists a union-closed sequence from $\mathcal{A}$ to $\mathcal{F}$.
\end{theorem}
\begin{proof}
Let $\mathcal{A}_0, \mathcal{A}_1 \dots, \mathcal{A}_t$ a sequence from $\mathcal{A}$ to $\mathcal{F}$.

If  $B(\mathcal{A}_{r-1})$ is not a subset of $\mathcal{F}$  then we can always choose $X_r  \in B(\mathcal{A}_{r-1})-\mathcal{F}$, so that $\mathcal{A}_{r}=\mathcal{A}_{r-1}-\{X_r\}$  is  union-closed by Lemma 2  and $\mathcal{F} \subset\mathcal{A}_{r}$.

So the condition that may stop this process is when it exists $q$ such that $B(\mathcal{A}_{q-1})\subseteq \mathcal{F}$. 

In this case we must have that for all $k\geq q$ $ B(\mathcal{A}_{q-1})\subseteq  B(\mathcal{A}_{k})$ by Lemma 3.

As a consequence $ B(\mathcal{A}_{q-1})\subseteq  B(\mathcal{A}_{t})$.  Let $X \in \mathcal{A}_{q-1}-\mathcal{A}_t$, by Lemma 1 it exists $\mathcal{T}=\{T_1;\dots T_r\} \subset B(\mathcal{A}_{q-1})$ such that  $X =\cup \mathcal{T}$.

 But $ B(\mathcal{A}_{q-1})\subseteq B(\mathcal{A}_{t})$, by hypothesis $\mathcal{A}_t=\mathcal{F}$  union-closed set , then $X \in \mathcal{A}_t$, so we have found a contradiction.

In conclusion we can always choose $\{X_q\}$ so that $\mathcal{A}_q$ is a  union-closed set and in this way we can build an union-closed sequence from $\mathcal{A}$ to $\mathcal{F}$.
\end{proof}

A trivial example of union-closed sequence from $\mathcal{A}$ to $\mathcal{F}$ is a sequence in which the elements $X_1, \ldots X_t$ have the property that $|X_i|\leq |X_j|$ if $i<j$.
Indeed we  can build a union-closed sequence starting from the end.
Consider  $X_t$ such that  for all $i$ we have $|X_t|\geq|X_i|$, it must be $\mathcal{F}\cup \{X_t\}$ is union-closed, as for all $P\in \mathcal {F}$, $P\cup X_t \in \mathcal {F}\cup \{X_t\}$. 

Now we define $\mathcal{D}= \mathcal{A}-\mathcal{F}$, $\mathcal{D}$ is the set of the elements eliminated by $\mathcal{A}$ in a sequence from $\mathcal{A}$ to $\mathcal{F}$.

\begin{lemma}
For all $j \in \cup \mathcal{A}$, $\mathcal{ F}\cup \mathcal{D}^j$ is a union-closed set.
 \end{lemma}
\begin{proof}
Let $X,Y \in \mathcal{ F}\cup\mathcal{ D}^j$ we have two cases:
 \begin{enumerate}
	\item $X,Y \in \mathcal{ F}$, as $\mathcal{F}$ is a union-closed $X\cup Y \in \mathcal{F}$ and then $X\cup Y \in \mathcal{ F}\cup\mathcal{ D}^j$.

	\item $X\in \mathcal{F}\cup \mathcal{D}^j$ and  $Y \in \mathcal{ D}^j$, as a consequence $j\in X\cup Y \notin (\mathcal{D}-\mathcal{D}^j)$ and then  $X\cup Y \in \mathcal{ F}\cup\mathcal{ D}^j$.
\end{enumerate}
\end{proof}

By the previous Lemma we can state the following definition

\begin{definition}
Let $i$ be such that $\mathcal{D}^i\neq \emptyset$.
A union-closed sequence from $\mathcal{A}$ to $\mathcal{F}$ defined by $\mathcal{D}= \{X_1; \dots X_{|\mathcal{D}|}\}$ is said  to be an \textbf{ ideal sequence} if and only if
\begin{enumerate}
          \item $\mathcal{A}_0=\mathcal{A}$
	\item for all $l\leq |\mathcal{D}|-|\mathcal{D}^i|$ we have $i \notin X_l$ and $\mathcal{A}_l= \mathcal{A}_{l-1}-\{X_l\}$.
         
	\item for all $s> |\mathcal{D}|-|\mathcal{D}^i|$ we have $i \in X_s $  and $\mathcal{A}_s= \mathcal{A}_{s-1}-\{X_s\}$.
          \item  $\mathcal{A}_{|\mathcal{D}|}=\mathcal{F}$
\end{enumerate}
\end{definition}

\begin{lemma}
For all union-closed sets $\mathcal{F}\subseteq \mathcal{A}$  there is an ideal sequence from $\mathcal{A}$ to $\mathcal{F}$.
 \end{lemma}
\begin{proof}

By Lemma 4, $\mathcal{ F}\cup \mathcal{D}^i$ is a union-closed set. Thus,  by Theorem 2, we can find an union-closed sequence from $\mathcal{A}$ to $\mathcal{F}\cup\mathcal{ D}^i$.

Now we can build a union-closed sequence from $\mathcal{F}\cup \mathcal{D}^i$ to $\mathcal{F}$. Suppose $ \mathcal{D}^i=\{X_1;\ldots X_r\}$  with the property $|X_{p-1}|\leq |X_p|$. 
 For all $s> |\mathcal{D}|-|\mathcal{D}^i|$ we have $\mathcal{A}_{|\mathcal{D}|-|\mathcal{D}^i|+1}= \mathcal{F}\cup \mathcal{D}^i-\{X_1\}$ and $\mathcal{A}_s= \mathcal{A}_{s-1}-\{X_{s}\}$.
 
Suppose for the sake of contradiction that in this sequence, $\mathcal{A}_v$ is not union-closed: we can find $Y, T \in \mathcal{A}_v$ such that $Y\cup T =X_v\notin \mathcal{A}_v$, obviously we must have $|Y|, |T|< |X_v|$.  

By construction of sequence for all $k>v$ $ |X_k|\geq |X_v|$, then  $Y, T\in \mathcal{F}$ and $\mathcal{F}$ is not union-closed, hence there is a contradiction. 
As a consequence we have described a way to build an ideal sequence.
\end{proof}

\begin{definition}
For all sets $Y$ and $X$ the set $E_X(Y)=\{T\in \mathcal{A}|  Y \subseteq  T \subseteq Y\cup X\}$ is called  \textbf{extension} of $Y$ on $X$.
\end{definition}

\begin{definition}
$X\in \mathcal{D}$ is said to be  \textbf{vincolated} if $\mathcal{F}\cup \{X\}$ is not union closed. In other words $X \in\mathcal{D}$ is said vincolated if $\exists Y \in \mathcal{F}$ such that $X \cup Y \in \mathcal{D}$.
\end{definition}

\begin{definition}
$X\in \mathcal{D}$ is said to be vincolated  to $Y\in \mathcal{D}$ if $\mathcal{F}\cup \{X\}$ is not union closed, but  $\mathcal{F}\cup\{X; Y\}$ is union closed. 
\end{definition}

\begin {theorem}
Let $\mathcal{F}\subseteq \mathcal{A}$ be a union-closed set, let $\mathcal{D}=\mathcal{A}-\mathcal{F}$, then if each $X\in \mathcal{D}-\mathcal{D}^i$ is vincolated, then it is possible to find an element $Y\in \mathcal{D}-\mathcal{D}^i$ vincolated to a non-vincolated $R\in\mathcal{D}^i$.
\end {theorem}
\begin{proof}

 Let $Y\in \mathcal{D}-\mathcal{D}^i$ of maximum cardinality on $\mathcal{D}-\mathcal{D}^i$. 
By hypothesis  $Y$ is vincolated, then it is possible to find $X \in \mathcal{F}$ such that $Y \cup X=R\in \mathcal{D}$. 
Moreover, as $Y$ has maximum cardinality on $\mathcal{D}-\mathcal{D}^i$, we must have $R\in \mathcal{D}^i$.
 As a consequence $E_{X}(Y) $ must be a subset of $\mathcal{D}$, otherwise $\mathcal{F}$ is not union-closed.
Indeed for all $T \in E_{X}(Y)$ we have $T \cup X=R$.
As we suppose $Y$ has maximum cardinality we must conclude $Y\cup \{i\}=R$.

Moreover we have to observe that $R$ could not be the set of maximum cardinality in $\mathcal{D}^i$. In this case for the sake of contradiction suppose that $ R$ could be vincolated. It means that we must have $K\in\mathcal{ D}^i$, then we can find $P\in \mathcal{F}$ such that $P \cup R=K$. But $\mathcal{F}$ is union-closed, then $P \cup X = P' \in \mathcal{F}$, then $Y \cup P' =K$ and $E_{P'}(Y) \subseteq \mathcal{D}$. As a consequence there must be an element $Y^*\in \mathcal{D}-\mathcal{D}^i$ such that $|Y^*|>|Y|$ and we find a contradiction. 

As a conclusion $R$ is not vincolated and  $\mathcal{F}\cup \{R\}$ is union-closed. 
We want to prove that $\mathcal{F}\cup \{R; Y\}$ is union-closed. 
Suppose it is not. It would imply that $Y$ is vincolated to  $K \in\mathcal{ D}^i$ with $Y\subset R=Y\cup\{i\} \subset K$, then, by definition,  we could find $S\in \mathcal{F}\cup \{R\}$ such that $Y \cup S=K\in \mathcal{D}$.
We have two cases:
\begin{enumerate}
\item $S\neq R$ thus $R\cup S=K\in \mathcal{D} $, then $R$ is vincolated to $K$, that would be a contradiction.
\item$S=R$ thus $K=R$ and we find a contradiction because $R\in \mathcal{F}\cap \mathcal{D}=\emptyset$.
\end{enumerate}
\end{proof}

\begin {definition}
Let $\mathcal{D}=\mathcal{A}-\mathcal{F}$ and let $\mathcal{D}^i\neq\mathcal{D}$. A sequence $\mathcal{A}_0, \mathcal{A}_1 \dots, \mathcal{A}_t$  from $\mathcal{A}$ to $\mathcal{F}$ is said  \textbf{optimal for $i$}  if and only if 
\begin{enumerate}
\item $i\in X_t$,
\item $i \notin X_{t-1}$  
\item Let $\mathcal{A}_0, \mathcal{A}_1 \dots,\mathcal{ A}_{t-2}$ is an ideal sequence from $\mathcal{A}$ to $\mathcal{F}\cup\{X_t; X_{t-1}\}.$ 
\end{enumerate}
\end{definition}

\begin{lemma}
 If $\mathcal{D}^i\neq \mathcal{D}$, it is always possible to build an optimal sequence for $i$ on $\mathcal{D}$.
\end{lemma}
\begin{proof}

We consider two different cases:
\begin{enumerate}
\item  if each $X\in \mathcal{D}-\mathcal{D}^i$ is vincolated, then by Theorem 3 it is possible to find an element $Y\in \mathcal{D}-\mathcal{D}^i$ vincolated to a non-vincolated $R\in\mathcal{D}^i$.
Thus we can build an  ideal sequence $\mathcal{A}_0, \mathcal{A}_1 \dots,\mathcal{ A}_{t-2}$  from $\mathcal{A}$ to $\mathcal{F}\cup\{Y; R\}$ and an optimal sequence  from $\mathcal{A}$ to $\mathcal{F}$ with $R= X_t$ and $Y= X_{t-1}$.

\item Otherwise it means that it exists an element $X^*\in \mathcal{D}-\mathcal{D}^i$ not vincolated, that means $\mathcal{F}\cup \{X^*\} $ is union closed.

 As $\mathcal{F}\cup \{X^*\} $ is union closed, by Lemma 5 we have   an  ideal sequence $\mathcal{A}_0, \mathcal{A}_1 \dots,\mathcal{ A}_{t-1}$  from $\mathcal{A}$ to $\mathcal{F}\cup \{X^*\}$, note that we must have  $\mathcal{F}\cup \{X^*; X_{t-1}\}$ union closed with $i\in X_{t-1}$ by definition of ideal sequence.

As a consequence consider the ideal sequence   $\mathcal{A}_0, \mathcal{A}_1 \dots,\mathcal{ A}_{t-2}$  from $\mathcal{A}$ to  $\mathcal{F}\cup \{X^*; X_{t-1}\}$.

We set $\mathcal{ A'}_{t-1}=\mathcal{ A}_{t-2}- \{X^*\}$. $\mathcal{ A'}_{t-1}$ must be union closed.   If $\mathcal{ A'}_{t-1}$ were not union closed,  as we know  $\mathcal{F}$ is union closed, we would have $S\in \mathcal{F}$ such that $S\cup X_{t-1}=X^*$. That is impossible as $i \in X_{t-1}$ and $i\notin X^*$.

At the end we set $\mathcal{ A'}_{t}=\mathcal{ A'}_{t-1}- \{X_{t-1}\}$. $\mathcal{ A'}_{t}=\mathcal {F}$ then it is union closed.

In this way the sequence  $\mathcal{A}_0, \mathcal{A}_1 \dots,\mathcal{ A}_{t-2}, \mathcal{ A'}_{t-1}, \mathcal{ A'}_{t} $ is an optimal sequence  from $\mathcal{A}$ to  $\mathcal{F}$.
\end{enumerate}
\end{proof}

\begin {definition}
Let $\mathcal{D}=\mathcal{A}-\mathcal{F}$, $i$ is said  \textbf{quasiminimal} on $\mathcal{D}$  for  a set   $\mathcal{Y}=\{Y_1;Y_2\}\subset \mathcal {F}$ if and only if
\begin{enumerate}
\item $i$ belongs to $Y_2$
\item $ i$ does not belong to $Y_1$
\item   $ i$ is minimal on  $\mathcal{D} \cup\mathcal{Y}$
\item there exists an optimal sequence  $\mathcal{A}_0, \mathcal{A}_1 \dots, \mathcal{A}_{t+2}$  from $\mathcal{A}$ to $\mathcal{F} -\mathcal {Y}$ such that $\mathcal{A}_{t+1}=\mathcal{A}_t-\{Y_1\}$ and $\mathcal{A}_{t+2}=\mathcal{A}_{t+1}-\{Y_2\}$. As the optimal sequence ends with $\mathcal{F} -\mathcal {Y}$, this set is union-closed.
\end{enumerate}  
\end{definition}

\begin {theorem}
For all $\mathcal{D}$, it exists $i$ such that if it exists a set $\mathcal{Y}=\{Y_1;Y_2\}$  $i$  quasiminimal on $\mathcal{D}$  for  $\mathcal{Y}$, then $2|(\mathcal{D}\cup\mathcal{Y})^i|\leq |(\mathcal{D}\cup\mathcal{Y})|+1.$

\end {theorem}
\begin{proof}

By induction on $|\mathcal{D}|$.
If  $|\mathcal{D}|=0$, for some $i$, if we can choose a set $\mathcal{Y}$, by definition of quasiminal we have $2|(\mathcal{D}\cup\mathcal{Y})^i|\leq |(\mathcal{D}\cup\mathcal{Y})|+1.$

If  $|\mathcal{D}|=1$: as $\mathcal {F}$ is union closed, we must have for some $j$ $\mathcal{D}=\{\{j\}\}$. As a consequence for all $i$ we can choose  a set $\mathcal{Y}$  $i$  quasiminimal on $\mathcal{D}$  for  $\mathcal{Y}$ we have $2|(\mathcal{D}\cup\mathcal{Y})^i|\leq |(\mathcal{D}\cup\mathcal{Y})|+1.$

If  $|\mathcal{D}|=2$, as $\mathcal {F}$ is union closed,  for some $j; k$ we must have $\mathcal{D}=\{\{k\};\{j\}\}$ or $\mathcal{D}=\{\{j\};\{k,j\}\}$. As a consequence for all $i$ we can choose  a set $\mathcal{Y}$  $i$  quasiminimal on $\mathcal{D}$  for  $\mathcal{Y}$ we have $2|(\mathcal{D}\cup\mathcal{Y})^i|\leq |(\mathcal{D}\cup\mathcal{Y})|+1.$

In general, let suppose the theorem true for $|\mathcal{D}|\leq t$ and we prove it for $|\mathcal{D}|=t+1$.
Particularly given a sequence  from $\mathcal{A}$ to $\mathcal{F}$, we indicate $\mathcal{D}_{j}=\{X_1, \ldots X_{j}\}$.
Consider two different cases:

\begin{enumerate}
\item Suppose that  for some $i$ that satisfay the theorem for  $\mathcal{D}_{t-1}$, $|\mathcal{D}^i|<|\mathcal{D}|$,  and by Lemma 6 we have $\mathcal{A}_0, \mathcal{A}_1 \dots,\mathcal{ A}_{t+1}$ optimal sequence for $i$.  If we could choose $\mathcal{Y}$ such that $i$  \textbf{quasiminimal} on $\mathcal{D}$  for  $\mathcal{Y}$, we could do as follows.
We split $\mathcal{D}_{t+1}=\mathcal{D}_{t-1}\cup\{X_t;X_{t+1}\}$, such that $i\notin X_{t}$ and $i\in X_{t+1}$. By definition $i$ quasiminimal on $\mathcal{D}_{t-1}$. Thus, by inductive hypothesys, $2|(\mathcal{D}_{t-1}\cup\{X_t;X_{t+1}\}\cup\mathcal{Y})^i|\leq |(\mathcal{D}\cup\mathcal{Y})|+1.$

\item Suppose $i$ satisfay the theorem for  $\mathcal{D}_{t-1}$  and $|\mathcal{D}^i|=|\mathcal{D}|$. Suppose  a  set $\mathcal{Y}$ exists such that $i$ quasiminimal on $\mathcal{D}$  for $\mathcal{Y}$. That means $i$ is minimal on $\mathcal{D}\cup\{Y_1;Y_2\}$, thus $i$ must be minimal in $\mathcal{D}\cup\{Y_1\}$,   as a result we have by definition $i$ quasiminimal on $\mathcal{D}_{t}$ for $\{Y_1;X_{t+1}\}$. 

Observe that $\mathcal{A}-(\mathcal{D}_t\cup\{Y_1\})$ is union-closed, because $\mathcal{A}-(\mathcal{D}_t\cup\{Y_1;X_{t+1}\})$ is union-closed and $i\in X_{t+1}$. On the other words for all $T\in \mathcal{A}-(\mathcal{D}_t\cup\{Y_1\})$. $T\cup X_{t+1} \in \mathcal{A}-(\mathcal{D}_t\cup\{Y_1\}).$
 
 As conclusion,    $2|(\mathcal{D}_{t+1}\cup \{Y_1;Y_2\})^i|=  2|(\mathcal{D}_{t}\cup \{Y_1;X_{t+1}\} \cup \{Y_{2}\})^i|= 2|\mathcal{D}^i_{t}|+4$. 
Moreover, by inductive hypothesys $2|(\mathcal{D}_{t}\cup \{Y_1;X_{t+1}\} \cup \{Y_{2}\})^i| \leq|\mathcal{D}_{t}|+1+2+2 $.As a result, as $2|\mathcal{D}^i_{t}|+4\leq |\mathcal{D}_{t}|+1+2+2$ and  $|\mathcal{D}_{t}^i| =|\mathcal{D}_{t}|$, we have $|\mathcal{D}_{t}|=1$ and we have just proved it at the inductive step.

\end{enumerate}
 \end{proof}

\begin {theorem}
Let $\mathcal{D}=\mathcal{A}-\mathcal{F}$ and let $j$ minimal on $\mathcal{D}$, with $|\mathcal{D}|>1$ . Then $2|\mathcal{D}^j|\leq |\mathcal{D}|+1$.
\end {theorem}
\begin{proof}
First of all we define $j$ minimal on $\mathcal{D}$ if and only if for all $k$ $|\mathcal{D}^{j}|\leq|\mathcal{D}^{k}|$.
As  $\mathcal{D}\neq\{\{j\}\}$ by hypothesys, $|\mathcal{D}^j|\neq |\mathcal{D}|$. 

By Lemma 6 it always possible to build an optimal sequence for $j$  on $\mathcal{D}$. 
Let  $\mathcal{A}_0, \mathcal{A}_1 \dots, \mathcal{A}_{t+1}$  this optimal sequence from $\mathcal{A}$ to $\mathcal{F}$, by definition we must have  
\begin{enumerate}
\item $j$ belongs to $X_{t+1}$
\item $ j$ does not belong to $X_t$
\item   $j$ is minimal on  $\mathcal{D}_{t-1} \cup\{X_{t};X_{t+1}\}$
\item it exists an optimal sequence  $\mathcal{A}_0, \mathcal{A}_1 \dots, \mathcal{A}_{t+1}$  from $\mathcal{A}$ to $\mathcal{F}$. 
\end{enumerate}

By Theorem 4, it exists $i$, not necessarily different from $j$,  such that for the set $\{X_{t};X_{t+1}\}$  $i$  quasiminimal on $\mathcal{D}_{t-1}$  for  $\{X_{t};X_{t+1}\}$ and $2|(\mathcal{D}_{t-1}\cup \{X_{t};X_{t+1}\}^i|\leq |\mathcal{D}_{t-1}\cup\{X_{t};X_{t+1}\}|+1.$

As a consequence $2|\mathcal{D}^i|\leq |\mathcal{D}|+1.$ and by definition of minimality of $j$ and quasiminimality of $i$ we have $2|\mathcal{D}^j|=2|\mathcal{D}^i|\leq |\mathcal{D}|+1.$
 
\end{proof}

Now we can prove Theorem 1

\begin{proof}
We just remember that $|\mathcal{A}|=2^n-1$ and $|\mathcal{A}^i|=2^{n-1}$.

$$|\mathcal{F}|=|\mathcal{A}|-|\mathcal{D}|=2|\mathcal{A}^i|-1-|\mathcal{D}|\leq 2|\mathcal{A}^i|-2|\mathcal{D}^i|=2|\mathcal{F}^i|$$
\end{proof}


\end{document}